\documentclass[a4paper,  leqno,  myheadings,  11pt,  twoside]{article}
\usepackage{mathrsfs}
\usepackage{hyperref}

\usepackage{amsmath,  amsthm,  amssymb,  amscd,  amsxtra, graphicx}
\usepackage{latexsym,  amsfonts}

\newcommand{\norm}[1]{\lVert#1\rVert}
\newcommand{\lnorm}[1]{{\lVert#1\rVert}_\infty}
\def\qc{quasi\-conformal }

\def\mc{\mathbb{C}}

\def\msu{\mathscr{U}}

\def\mcn{\mathcal{N}}

\def\T{Teich\-m\"ul\-ler }
\def\nde{non-decreasable }
\def\dcs{decreasable }

\def\bmu{\emu_Z}

\def\nmu{\|\mu\|_\infty}

\def\nmu{\|\mu\|_\infty}

\def\vp{\varphi}

\def\vp{\varphi}

\def\pa{\partial}

\def\ov{\overline}

\def\de{\Delta}
\def\q1s{Q^1(S)}

\def\ts{T(S)}

\def\zs{Z(S)}

\def\zde{Z(\de)}

\def\tde{T(\de)}

\def\emu{[\mu]}

 \makeatletter
\@addtoreset{equation}{section}

\newtheorem{theorem}{Theorem}

\newtheorem{lemma}{Lemma}[section]

\newtheorem{theo}{Theorem}

\newtheorem{problem}{Problem}

\newtheorem*{note}{Note}

\renewcommand{\thetheoa}

\begin{document}

\title{\bf{On infinitesimally weakly non-decreasable Beltrami differentials}
\author{GUOWU YAO
\\
Department of Mathematical Sciences,
  Tsinghua University\\ Beijing,  100084,
   People's Republic of
  China \\E-mail: \texttt{wallgreat@mail.tsinghua.edu.cn}
}}
 \date{}
\maketitle
\begin{abstract}\noindent
Z. Zhou  et al. proved that in a \T equivalence class, there exists an extremal \qc mapping with a weakly \nde dilatation.    In this paper, we prove that in an infinitesimal  equivalence class, there exists a weakly  \nde extremal Beltrami differential.
\end{abstract}
\renewcommand{\thefootnote}{}

\footnote{
\qc map, weakly non-decreasable,  non-decreasable.}

\footnote{The  work was  supported by   the National Natural Science Foundation of China (Grant
No. 11771233).}

\section{\!\!\!\!\!{. }
Introduction}\label{S:intr}

Suppose that ${S}$ is a Jordan domain  in the complex plane
$\mathbb{C}$,  and let $w=f(z)$ be   a \qc mapping on ${S}$. The Beltrami differential  of $f$ is defined by
\begin{equation*}
\mu(z)=\frac{f_{\overline z}(z)}{f_z(z)}\frac{d\overline z}{dz},
\end{equation*} which is also called the complex dilatation of
$f$.

Denote by $Bel(S)$ the Banach space of Beltrami differentials
$\mu=\mu(z)d\bar z/dz$ on $S$ with finite $L^{\infty}$-norm and by
$M(S)$ the open unit ball in $Bel(S)$.  Let $z_1,z_2,z_3$  be three boundary
points on $\pa{S}$. For a given $\mu\in M({S})$,
denote by $f^\mu$ the uniquely determined \qc mapping of
${S}$ onto itself with complex dilatation $\mu$ and
normalized to fix $z_1,z_2,z_3$. Two elements $\mu$ and $\nu$
in $M({S})$ are \T equivalent, which is denoted by
$\mu\sim\nu$, if
$f^\mu|_{\pa{S}}=f^\nu|_{\pa{S}}$. Then
$T({S})=M({S})/\sim $ is the  \T space of
${S}$. The equivalence class of the Beltrami differential
zero is the basepoint of $T({S})$.

For $\mu\in M(S)$, define
\begin{equation*}
k_0(\emu)=\inf\{\|\nu\|_\infty:\,\nu\in\emu\}.
\end{equation*}
We say that $\mu$ is extremal  in $\emu$
  if $\nmu=k_0(\emu)$ (the corresponding \qc map $f$ is said to be extremal for its boundary values as well), and uniquely extremal if $\|\nu\|_\infty>k_0(\mu)$ for any other
$\nu\in\emu$.

Let $Q(S)$ denote the  Banach space
 of integrable holomorphic quadratic differentials on $S$ with
$L^1-$norm
\begin{equation*}
\|\vp\|=\iint_{S}|\vp(z)|\, dxdy<\infty.
 \end{equation*}
 In what follows,  let $\q1s$
denote the unit sphere of $Q(S)$.

 Two Beltrami differentials $\mu$ and $\nu$ in $Bel(S)$ are said to
be infinitesimally equivalent if
\begin{equation*}\iint_S(\mu-\nu)\vp \, dxdy=0,  \text{
for any } \vp\in Q(S).
\end{equation*}
The tangent space $\zs$ of $\ts$ at the basepoint is defined as  the quotient space of $Bel(S)$ under the equivalence
relation.  $Q(S)$ is predual to $\zs$.  Denote by $\bmu$ the equivalence class of $\mu$ in
$\zs$.  In particular, we use $\mcn(S)$ to denote the set of Beltrami differentials in $Bel(S)$ that is equivalent to 0. The elements in $\mcn(S)$ are also called infinitesimally trivial.

$\zs$ is a Banach space and  its standard
norm satisfies
\begin{equation*}
\|\bmu\|:=\sup_{\vp\in \q1s}Re\iint_S \mu\vp \,
dxdy=\inf\{\|\nu\|_\infty:\,\nu\in\bmu\}.\end{equation*}
 We say that $\mu$ is extremal  (in $\bmu$) if $\nmu=\|\bmu\|$, and
uniquely extremal if $\|\nu\|_\infty>\nmu$ for any other $\nu\in
\bmu$.

 A Beltrami differential $\mu$ (not necessarily extremal) is called to be  \emph{\nde} in its class $\emu$ if  for $\nu\in \emu$,
  \begin{equation}
  |\nu(z)|\leq|\mu(z)|\  a.e. \text{\ in } S,\end{equation}
implies that $\mu=\nu$; otherwise, $\mu$  is called to be \emph{decreasable}.

 The notion of \nde dilatation was firstly introduced by Reich in \cite{Re4} when he studied the unique extremality of \qc mappings.
 A uniquely extremal Beltrami differential is obviously non-decreasable.

 Let $\de$ denote the unit disk $\{z\in\mc:\;|z|<1\}$ and $\tde$ be  the universal \T space.
 In \cite{SC}, Shen and Chen   proved  the following theorem.
 \begin{theo}\label{Th:sc}
For every  \qc mapping $f$ from $\de$ onto itself, there exist infinitely many \qc mappings $g$ in the \T equivalence class $[f]$ of $\tde$  each of which has  a  \nde dilatation unless $[f]$ contains a conformal mapping.
\end{theo}

 The author \cite{Yao2} proved that a \T class may contain infinitely many \nde extremal dilatations. The existence of a \nde extremal in a class is generally unknown.

In \cite{ZZC}, Zhou et al. defined \emph{weakly \nde dilatation} as follows. Let $\mu\in M(S)$. $\mu$  is called a \emph{strongly \dcs dilatation} in $\emu$ if there exists $\nu\in \emu$ satisfying the following conditions:\\
(A) $|\nu(z)|\leq |\mu(z)|$ for almost all $z\in S$,\\
(B) There exists a domain $G\subset S$  and a positive number $\delta>0$ such that
\begin{equation*}
|\nu(z)|\leq |\mu(z)|-\delta, \text{ for almost all  } z\in G.
\end{equation*}
Otherwise, $\mu$ is called \emph{weakly non-decreasable}. In other words, a Beltrami differential $\mu$ is called weakly \nde if either $\mu$ is \nde or $\mu$ is decreasable but is not strongly decreasable.

Let $\mu\in Bel(S)$. The definition that $\mu$ is a strongly decreasable dilatation in $\bmu$ almost takes word-by-word from that  $\mu$ is a strongly decreasable dilatation in $\emu$ except that the equivalence class $\emu$ is replaced by $\bmu$, so is the definition that $\mu$ is a weak non-decreasable dilatation in $\bmu$.

In \cite{ZZC}, Zhou et. al proved the following theorem.
\begin{theo}\label{Th:zzc}
For every extremal \qc mapping $f$ from $\de$ onto itself, there exists an extremal \qc mapping $g$ in the \T equivalence class $[f]$  with a weakly \nde dilatation.
\end{theo}

The goal of the paper is to prove the counterpart of Theorem \ref{Th:zzc} in the infinitesimal setting.

\begin{theorem}\label{Th:infweaknde}Suppose  $\bmu\in \zde$.  Then
 there is a weakly \nde extremal dilatation $\nu$ in $\bmu$.
\end{theorem}

After doing some preparations in Section \ref{S:prepar}, we prove Theorem \ref{Th:infweaknde} in Section \ref{S:proof}.

\section{\!\!\!\!\!{. }
Some preparations }\label{S:prepar}

\begin{lemma}\label{Th:reich}
Let  $\nu\in Bel(\de)$. Then for any given $\epsilon>0$, there exists some $r\in (0,1)$  and  $\mu\in Bel(\de)$ such that \\
(1) $\mu\in \mcn(\de),$ (2) $\mu(z)=\nu(z)$, $z\in \de_r$, (3) $\norm{\mu|_{U_r}}_\infty<\epsilon$,\\
where $\de_r=\{z\in \de:|z|<r\}$, $r\in (0,1)$ and $U_r=\de\backslash \de_r$.
\end{lemma}
\begin{proof}
By Theorem 1.1 in \cite{Re2}, there exists a unique function $\beta(z)$, holomorphic in $\mc\backslash \ov{\de_r}$, such that
\begin{equation*}
\mu(z)=\begin{cases}\nu(z),  \quad z\in \de_r,\\
\beta(z),\quad z\in U_r,
\end{cases}
\end{equation*}
belongs to $\mathcal{N}(\de)$; namely,
\begin{equation*}
\beta(z)=-\frac{z}{\pi (1-r^2)}\iint_{\de_r}\frac{\nu(\zeta)}{\zeta-z}d\xi d\eta,\; z\in \mc\backslash \ov{\de_r}.
\end{equation*}
To complete the proof of this lemma, it is sufficient to show that $\norm{\beta|_{U_r}}_\infty<\epsilon$ for small $r>0$.
We need to evaluate $|\beta(z)|$ for $z\in U_r$. Let $z'$ be the intersection point of the segment $\ov{oz}$ with the circle $\{|\zeta|=r\}$ and $B_r=\{|\zeta-z'|<\frac{r}{2}\}$. Then  for $z\in U_r$,
\begin{align*}
&|\beta(z)|=\frac{|z|}{\pi (1-r^2)}\left|\iint_{\de_r}\frac{\nu(\zeta)}{\zeta-z}d\xi d\eta\right|\leq\frac{\norm{\nu}_\infty}{\pi (1-r^2)}\iint_{\de_r}\frac{1}{|\zeta-z|}d\xi d\eta\\
&\leq M\iint_{\de_r}\frac{1}{|\zeta-z'|}d\xi d\eta\leq M\left(\iint_{B_r}\frac{1}{|\zeta-z'|}d\xi d\eta+\iint_{\de_r\backslash B_r}\frac{1}{|\zeta-z'|}d\xi d\eta\right)\\
&\leq M\left(\pi r+\frac{2}{r}\iint_{\de_r}d\xi d\eta\right)=3M\pi r=\frac{3\pi\norm{\nu}_\infty r}{\pi (1-r^2)},
\end{align*}
where $M=\frac{\norm{\nu}_\infty}{\pi (1-r^2)}$. This lemma follows readily.

\end{proof}

Throughout the paper, we denote by $\de(\zeta,r)$ the round disk $\{z:\;|z-\zeta|<r\}$ ($r>0$).

\begin{lemma}\label{Th:cor}Suppose that $\mu$ and  $\alpha$ belong to $ Bel(\de)$. Then for any $\zeta\in \de$ and $\epsilon>0$,  there exists  $\nu\in \bmu$ and a small $r>0$ such that $\lnorm{\nu|_{\de\backslash\de(\zeta,r)}}\leq \lnorm{\mu}+\epsilon$ and
\begin{equation}\label{Eq:mug0} \nu(z)=\alpha(z),\quad \text{ when } z\in \de(\zeta,r)=\{z:\;|z-\zeta|<r\}.\end{equation}
 In particular, $\nu$ vanishes on $\de(\zeta,r)$ when $\alpha=0$.
\end{lemma}

\begin{proof}
Choose sufficiently small $\rho>0$ such that the disk $D=\Delta(\zeta,\rho)$ is contained in $\de$. Restrict $\mu$ on $D$.
Applying  Lemma \ref{Th:reich} to $D$, we can find  some small $r\in (0,\rho)$ and $\chi\in\mcn(D)$ such that $\chi=\alpha-\mu$ on $\de(\zeta,r)$ and $\norm{\chi}_\infty<\epsilon$ on $D\backslash \de(\zeta,r)$. Put
\begin{equation*}
\nu(z)=\begin{cases}\chi(z)+\mu(z),  \quad &z\in D,\\
\mu(z), & z\in \de\backslash D.
\end{cases}
\end{equation*}
Then $\nu\in\bmu$ and $\nu=\alpha$ in $\de(\zeta,r)$.  It is clear that  $\lnorm{\nu|_{\de\backslash\de(\zeta,r)}}\leq \lnorm{\mu}+\epsilon$.
\end{proof}

\section{\!\!\!\!\!{. }
Proof of  the main result}\label{S:proof}

Theorem \ref{Th:infweaknde} follows from the following generalized theorem immediately.

\begin{theorem}\label{Th:lempseudo}Suppose  $\bmu\in \zde$.  Let $\chi\in \bmu$ and  $\msu=\{\alpha\in \bmu:\; |\alpha(z)|\leq |\chi(z)| \text{ a.e. on } \de\}$. Then
then there is a weakly \nde dilatation $\nu$ in $\msu$.
\end{theorem}

\begin{proof} The technique of the proof is partly taken from \cite{ZZC}. We use it in a refined way.

   If $\chi$ is a weakly \nde dilatation, then  $\nu=\chi$ is the desired  weakly \nde dilatation. Otherwise, $\chi$ is strongly decreasable. By definition, there exists a Beltrami differential $\eta\in\msu$  such that \\
(1) $|\eta(z)|\leq |\chi(z)|$ for almost all $z\in \de$,\\
(2) there exists a small round disk $\de(z'_0,s'_0)\subset \de$  and a positive number $\delta>0$ such that
\begin{equation*}
|\eta(z)|\leq |\chi(z)|-\delta, \text{ for almost all  } z\in \de(z'_0,s'_0).
\end{equation*}
  Regard
  $[\eta|_{\de(z'_0,s'_0)}]$ as the point in the infinitesimal  space $Z(\de(z'_0,s'_0))$. Applying Lemma \ref{Th:cor} for $\eta$ on $\de(z'_0,s'_0)$, we can find a Beltrami differential $\eta'\in [\eta|_{\de(z'_0,s'_0)}]_Z$ such that
  \[\lnorm{\eta'}\leq\lnorm{\eta|_{\de(z'_0,s'_0)}}+\frac{\delta}{2},  \]
and $\eta'(z)=0$ on some small disk $\de(z'_0,r'_0)$ ($r'_0<s'_0$).

Set
    \begin{equation*}\label{Eq:land1}
\chi'_0(z)=\begin{cases}\eta(z),\quad &z\in \de\backslash  \ov {\de(z'_0,s'_0)},\\
\eta'(z), \quad & z\in \de(z'_0,s'_0).
\end{cases}
\end{equation*}
  Then   $\chi'_0\in \msu$ and  $\chi'_0(z)=0$ on   $\de(z'_0,r'_0)$.

  Let $\Lambda_0$ denote the collection of $\alpha\in \bmu$ with the following conditions:\\
(a) $ |\alpha(z)|\leq |\chi'_0(z)|$  a.e.  on   $\de$,\\
(b) there exists some small disk $\de(\zeta,r)\subset \de$  such that $\alpha(z)=0$ on $\de(\zeta,r)$.

  It is obvious that $\Lambda_0\subset \msu$ and $\chi'_0\in\Lambda_0$. If $\alpha\in \Lambda_0$, let
\[\rho_0(\alpha)=\sup\{r:\;\alpha(z)=0 \text{ on some }\de(\zeta,r)\subset \de\}.\]
Put
\[\rho_0=\sup\{\rho_0(\alpha):\; \alpha\in \Lambda_0\}.\]

   We proceed with the construction of  a sequence of Beltrami differentials $\{\chi_n\}$ in $\msu$.

  $n=0)$ Choose $\chi_0\in \Lambda_0$ such that $\rho_0(\chi_0)\geq r_0:=\frac{\rho_0}{2}$ and $\chi_0(z)=0$ on some $\de(z_0,r_0)\subset \de$.  If $\chi_0$ is a weakly \nde dilatation, then   $\nu=\chi_0$ is the desired  weakly \nde dilatation. Otherwise, $\chi_0$ is strongly decreasable. By definition, there exists a Beltrami differential $\eta_0\in\msu$  such that \\
(1) $|\eta_0(z)|\leq |\chi_0(z)|$ for almost all $z\in \de$,\\
(2) there exists a small round disk $\de(z'_1,s'_1)\subset \de$  and a positive number $\delta_0>0$ such that
\begin{equation*}
|\eta_0(z)|\leq |\chi_0(z)|-\delta_0, \text{ for almost all  } z\in \de(z'_1,s'_1).
\end{equation*}
Since $\chi_0(z)=0$ on $\de(z_0,r_0)$, it forces that $\de(z'_1,s'_1)\subset \de\backslash \de(z_0,r_0)$.
  Applying Lemma \ref{Th:cor} for $\eta_0$ on $\de(z'_1,s'_1)$, we can find a Beltrami differential $\chi'_1\in \msu$ such that $\chi'_1(z)=0$ on  some
  small disk $\de(z'_1,r'_1)$ ($r'_1<s'_1$).

  Let $\Lambda_1$ denote the collection of $\alpha\in \bmu$ with the following conditions:\\
(a) $ |\alpha(z)|\leq |\chi_0(z)|$  a.e.  on   $\de$,\\
(b) there exists some small disk $\de(\zeta,r)\subset \de\backslash \de(z_0,r_0)$  such that $\alpha(z)=0$ on $\de(\zeta,r)$.

It is obvious that $\Lambda_1\subset \msu$ and $\chi'_1\in\Lambda_1$. If $\alpha\in \Lambda_1$, let
\[\rho_1(\alpha)=\sup\{r:\;\alpha(z)=0 \text{ on some }\de(\zeta,r)\subset \de\backslash \de(z_0,r_0)\}.\]
Put
\[\rho_1=\sup\{\rho_1(\alpha):\; \alpha\in \Lambda_1\}.\]

 $n=1)$ Choose $\chi_1\in \Lambda_1$ such that $\rho_0(\chi_1)\geq r_1:=\frac{\rho_1}{2}$ and
  \[\chi_1(z)=0 \text{ on some } \de(z_1,r_1)\subset \de\backslash \de(z_0,r_0).\]
  If $\chi_1$ is a weakly \nde dilatation, then let $\nu=\chi_1$.  Otherwise, $\chi_1$ is strongly decreasable. By definition, there exists a Beltrami differential $\eta_1\in\msu$  such that \\
(1) $|\eta_1(z)|\leq |\chi_1(z)|$ for almost all $z\in \de$,\\
(2) there exists a small round disk $\de(z'_2,s'_2)\subset \de$  and a positive number $\delta_2>0$ such that
\begin{equation*}
|\eta_1(z)|\leq |\chi_1(z)|-\delta_1, \text{ for almost all  } z\in \de(z'_2,s'_2).
\end{equation*}
Since $\chi_1(z)=0$ on $\de(z_0,r_0)\cup \de(z_1,r_1)$, it forces that
\[\de(z'_2,s'_2)\subset \de\backslash (\de(z_0,r_0)\cup \de(z_1,r_1)).\]
  Applying Lemma \ref{Th:cor} for $\eta_1$ on $\de(z'_2,s'_2)$, we can find a Beltrami differential $\chi'_2\in \msu$ such that $\chi'_2(z)=0$ on  small disk $\de(z'_2,r'_2)$ ($r'_2<s'_2$).

  Let $\Lambda_2$ denote the collection of $\alpha\in \bmu$ with the following conditions:\\
(a)  $ |\alpha(z)|\leq |\chi_1(z)|$  a.e.  on   $\de$,\\
(b) there exists some small disk $\de(\zeta,r)\subset \de\backslash (\de(z_0,r_0)\cup \de(z_1,r_1))$  such that $\alpha(z)=0$ on $\de(\zeta,r)$.

It is obvious that $\Lambda_2\subset \msu$ and  $\chi'_2\in\Lambda_2$.
If $\alpha\in \Lambda_2$, let
\[\rho_2(\alpha)=\sup\{r:\;\alpha(z)=0 \text{ on some  }\de(\zeta,r)\subset \de\backslash (\de(z_0,r_0)\cup \de(z_1,r_1))\}.\]
Put
\[\rho_2=\sup\{\rho_2(\alpha):\; \alpha\in \Lambda_2\}.\]

  $n\to n+1)$ If we can choose a  weakly \nde dilatation $\chi_n\in \Lambda_n$, then let $\nu=\chi_n$.  Otherwise, proceeding as above, we find four sequences,  $\{\chi_n\in \msu\}$, $\{\Lambda_n\subset \msu\} $, $\{\rho_n\in (0,1)\},$ $\{\de(z_n,r_n)\subset\de\}$ ($r_n=\frac{\rho_n}{2}$) as follows.

   $\chi_n\in \Lambda_{n}$ satisfies $\rho_{n}(\chi_n)\geq r_n:=\frac{\rho_{n}}{2}$ and
  \[\chi_n(z)=0 \text{ on some } \de(z_n,r_n)\subset \de\backslash (\bigcup_{k=0}^{n-1}\de(z_k,r_k)).\]

 $\Lambda_{n+1}$ is  the collection of $\alpha\in \bmu$ with the following conditions:\\
(a) $ |\alpha(z)|\leq |\chi_n(z)|$  a.e.  on   $\de$,\\
(b) there exists some small disk $\de(\zeta,r)\subset \de\backslash (\bigcup_{k=0}^n\de(z_k,r_k))$  such that $\alpha(z)=0$ on $\de(\zeta,r)$.

Since $\chi_n$ is not a weakly \nde dilatation in $\msu$, $\Lambda_{n+1}$ is not void by the foregoing reason.
If $\alpha\in \Lambda_{n+1}$, let
\[\rho_{n+1}(\alpha)=\sup\{r:\;\alpha(z)=0 \text{ on }\de(\zeta,r)\subset \de\backslash (\bigcup_{k=0}^{n}\de(z_k,r_k))\}.\]
Put
\[\rho_{n+1}=\sup\{\rho_{n+1}(\alpha):\; \alpha\in \Lambda_{n+1}\}.\]

It is clear that
\begin{equation}\label{Eq:rhor}\lim_{n\to \infty}\rho_n=\lim_{n\to \infty}r_n=0.\end{equation}

  Let $B_n=\bigcup^n_{k=0}\de(z_k,r_k)$, $n=0,1,\cdots$. Then
  \[ B=\lim_{n\to\infty} B_n=\bigcup^\infty_{n=0}\de(z_n,r_n)\subset \de.\]

  By the weak-* compactness, there exists a subsequence of
$\{\chi_n\}$, still denoted by $\{\chi_n \}$, which converges to a
limit $\nu\in Bel(\de)$ in the weak-* topology, that is, for
any $\phi\in L^1(\de)$,
\begin{equation}\label{Eq:weak}
\lim_{n\to\infty}\iint_\de\chi_n(z)\phi(z)dxdy=\iint_\de\nu(z)\phi(z)dxdy.
\end{equation}
Now when $\phi\in Q(\de)$, since $\chi_n\in \bmu$, we have
\begin{equation*}
\iint_\de\chi_n(z)\phi(z)dxdy=\iint_\de\mu(z)\phi(z)dxdy.
\end{equation*}
We have
\[\iint_\de\nu(z)\phi(z)dxdy=\iint_\de\mu(z)\phi(z)dxdy\]
for all $\phi\in Q(\de)$ and hence $\nu\in\bmu$. On the other hand,
since $\chi_n$ converges to $\nu$ in the weak-* topology, it follows by the standard functional analysis theory
that\begin{equation*}
\|\nu\|_\infty\leq\liminf_{n\to\infty}\|\chi_n\|_\infty=k.
\end{equation*}

 \textit{Claim 1.}   $\nu\in \msu$ and $\nu(z)=0$ on $B$ by (\ref{Eq:str}).

 By the inductive construction, we see that $|\chi_{n+1}(z)|\leq |\chi_{n}(z)|$  a.e. on $\de$. Set
 \begin{equation*}\beta(z)= \lim_{n\to \infty}|\chi_{n}(z)|.\end{equation*}

 We now show that
 \begin{equation}\label{Eq:str}
|\nu(z)|\leq \beta(z),\; a.e.\;z\in \de.
 \end{equation}

 For a given $z\in  \de$ and sufficiently small $r>0$, we restrict our consideration on $\Delta(z, r)$.
Regard $\{\chi_n\}$ as a sequence in $L^1(\Delta(z, r))$. Since the dual space of $L^1(\Delta(z, r))$ is $L^\infty(\Delta(z, r))$, the weak-* convergence of $\{\chi_n\}$ in $L^\infty(\Delta(z, r))$ actually implies that $\chi_n$ converges to $\nu$ in $L^1(\Delta(z, r))$ in the weak topology. Thus, we have
 \begin{equation}\label{Eq:623} \begin{split}
\iint_{\Delta(z, r)}|\nu(\zeta)|d\xi d\eta\leq \varliminf_{n\to \infty}\iint_{\Delta(z, r)}|\chi_n(\zeta)|d\xi d\eta.
 \end{split}
  \end{equation}
So, we see that
 \begin{equation}\label{Eq:624} \begin{split}
&\iint_{\Delta(z, r)}|\nu(\zeta)|d\xi d\eta\leq \varlimsup_{n\to \infty}\iint_{\Delta(z, r)}|\chi_n(\zeta)|d\xi d\eta\\
&\leq \iint_{\Delta(z, r)}\varlimsup_{n\to \infty}|\chi_n(\zeta)|d\xi d\eta=\iint_{\Delta(z, r)}\beta(\zeta)d\xi d\eta.
\end{split}
  \end{equation}

Because for almost every $z\in  \de$, $z$ is a Lebesgue point for both $|\nu|$ and $\beta$. Therefore, it holds  for almost all $z\in  \de$ that

  \begin{equation}\label{Eq:625}
|\nu(z)|=\lim_{r\to0}\frac{\iint_{\Delta(z, r)}|\nu(\zeta)|d\xi d\eta}{\pi r^2}.
 \end{equation}

Furthermore, by (\ref{Eq:624})  we get that for almost all $z\in  \de$,
 \begin{equation}\label{Eq:625} \begin{split}
|\nu(z)|=\lim_{r\to0}\frac{\iint_{\Delta(z, r)}|\nu(\zeta)|d\xi d\eta}{\pi r^2}\leq
\lim_{r\to0}\frac{\iint_{\Delta(z, r)}\beta(\zeta)d\xi d\eta}{\pi r^2}=\beta(z).
\end{split}
  \end{equation}

Since  each  $\chi_n$ belongs to $\msu$ and $\chi_n(z)=0$ on $B_n$, we see that $\nu\in \msu$  and $\nu(z)=0$ on $B$ by (\ref{Eq:str}).

  \textit{Claim 2.} $\nu$ is a weakly \nde dilatation.

  Suppose to the contrary. Then there exists a Beltrami differential $\eta\in\msu$  such that \\
(1) $|\eta(z)|\leq |\nu(z)|$ for almost all $z\in \de$,\\
(2) there exists a small round disk $\de(\zeta,r')\subset \de$  and a positive number $\delta'>0$ such that
\begin{equation*}
|\eta(z)|\leq |\nu(z)|-\delta', \text{ for almost all  } z\in \de(\zeta,r').
\end{equation*}
Since $\nu(z)=0$ on $B$, it forces that $\de(\zeta,r')\subset \de\backslash B$.  Applying Lemma \ref{Th:cor} for $\eta$ on $\de(\zeta,r')$, we can find a Beltrami differential $\nu'\in \msu$ and  some small disk $\de(\zeta,r)\subset \de\backslash B$ ($r<r'$)   such that $\nu'(z)=0$ on $\de(\zeta,r')\cup B$.
It is obvious that $\nu'$ belongs to  $\bigcap_{n=0}^\infty\Lambda_n$.
However, by  (\ref{Eq:rhor}) it contradicts the choice of $\chi_n$. The claim is proved, and so is the theorem.

  \end{proof}

\section{\!\!\!\!\!{. }
Concluding remarks}\label{S:remarks}

In the end of the paper \cite{ZZC}, Zhou et al.  posed the following problem.

\begin{problem} Suppose that $\mu$ is a weakly \nde dilatation in $\emu$. Is
$\mu$ is necessarily a \nde  dilatation?\end{problem}
The author gave a negative answer to the problem in \cite{Yao10} and obtained further results in \cite{Yao11}. It is natural to pose 
the following problem.
\begin{problem}  Suppose that $\mu$ is a weakly \nde dilatation in $\bmu$. Is $\mu$  is a \nde one?
  \end{problem}
  We will show that the answer is also negative for this case in another paper \cite{Yao12}.\\

\noindent\textbf{Acknowledgements.}  The author would like to thank the referee for   his valuable comments and careful corrections   which improved the exposition.

\renewcommand\refname{\centerline{\Large{R}\normalsize{EFERENCES}}}

\end{document}